\documentclass[9pt,reqno]{amsart}

\usepackage[T1]{fontenc}
\usepackage{lmodern}
\usepackage{microtype}

\usepackage[margin=1in]{geometry}

\usepackage{amsmath,amssymb,amsfonts,amscd,amsthm}
\usepackage{mathtools}
\usepackage{mathrsfs}
\usepackage{bm}
\usepackage{esint}
\usepackage{aliascnt}

\usepackage{graphicx}
\usepackage{float}
\usepackage[font=small,labelfont=bf]{caption}
\usepackage{subcaption}

\usepackage{tikz}
\usetikzlibrary{calc}

\usepackage{pgfplots}
\usepgfplotslibrary{groupplots}
\pgfplotsset{compat=1.18}

\graphicspath{{figures/}{images/}}

\usepackage{enumitem}
\usepackage{comment}
\usepackage[normalem]{ulem}
\usepackage{xcolor}

\usepackage[
    colorlinks=true,
    linkcolor=blue!60!black,
    citecolor=blue!60!black,
    urlcolor=blue!60!black
]{hyperref}

\usepackage[nameinlink,capitalize,noabbrev]{cleveref}

\numberwithin{equation}{section}

\theoremstyle{plain}
\newtheorem{theorem}{Theorem}[section]

\newaliascnt{proposition}{theorem}
\newtheorem{proposition}[proposition]{Proposition}
\aliascntresetthe{proposition}

\newaliascnt{lemma}{theorem}
\newtheorem{lemma}[lemma]{Lemma}
\aliascntresetthe{lemma}

\newaliascnt{corollary}{theorem}
\newtheorem{corollary}[corollary]{Corollary}
\aliascntresetthe{corollary}

\newaliascnt{conjecture}{theorem}

\aliascntresetthe{conjecture}

\theoremstyle{definition}
\newaliascnt{definition}{theorem}

\aliascntresetthe{definition}

\newaliascnt{assumption}{theorem}

\aliascntresetthe{assumption}

\newaliascnt{example}{theorem}

\aliascntresetthe{example}

\theoremstyle{remark}
\newaliascnt{remark}{theorem}

\aliascntresetthe{remark}

\crefname{theorem}{Theorem}{Theorems}
\crefname{proposition}{Proposition}{Propositions}
\crefname{lemma}{Lemma}{Lemmas}
\crefname{corollary}{Corollary}{Corollaries}
\crefname{definition}{Definition}{Definitions}
\crefname{assumption}{Assumption}{Assumptions}
\crefname{example}{Example}{Examples}
\crefname{remark}{Remark}{Remarks}
\crefname{equation}{Equation}{Equations}
\crefname{section}{Section}{Sections}

\newcommand{\R}{\mathbb{R}}

\newcommand{\N}{\mathbb{N}}

\newcommand{\F}{\mathcal{F}}
\newcommand{\G}{\mathcal{G}}

\newcommand{\J}{\mathcal{J}}
\newcommand{\K}{\mathcal{K}}

\newcommand{\ep}{\varepsilon}
\newcommand{\D}{\Delta}
\newcommand{\p}{\partial}

\newcommand{\defeq}{\vcentcolon=}

\DeclarePairedDelimiter{\norm}{\lVert}{\rVert}

\DeclarePairedDelimiterX{\inner}[2]{\langle}{\rangle}{#1,#2}

\title{A Short Proof of Optimal Regularity for minimizers of the Alt-Phillips Problem}

\author{Kunyi (Mark) Ma}
\address{Department of Mathematics, Columbia University, New York, NY 10027}
\email{km4046@columbia.edu}

\date{\today}

\subjclass[2020]{Primary 35R35; Secondary 35J60, 49J40}

\keywords{Free boundary problem, Alt--Phillips functional, variational methods}

\begin{document}

\begin{abstract}
We give a self-contained short proof of optimal regularity for minimizers of the Alt-Phillips Free Boundary Problem for $\gamma \in (0, 1)$. We adopt a dichotomy argument that originates from \cite{DS20}.

\end{abstract}

\maketitle

\tableofcontents

\section{Introduction}

Throughout the article we let
\[
    \gamma\in(0,1),
    \qquad
    \beta=\frac{2}{2-\gamma}\in(1,2).
\]

Let $B_1\subseteq\R^d$ denote the unit ball, and let
$u\in H^1(B_1)$ be a minimizer of the Alt--Phillips functional
\[
    \J^\gamma(u,B_1)
    \defeq
    \int_{B_1}\bigl(|\nabla u|^2+u_+^\gamma\bigr)\,dx
\]
among functions with prescribed non-negative boundary data. It has been known
since the work of Phillips \cite{Phi83} that minimizers satisfy the optimal
regularity
\[
    u\in C^{1,\beta-1}_{\mathrm{loc}}(B_1).
\]
For related developments on Alt--Caffarelli--Phillips type free boundary
problems, we refer to \cite{Caf98} and \cite{PSU12} for the obstacle problem corresponding to $\gamma=1$; for the one-phase Bernoulli problem
corresponding to $\gamma=0$, see \cite{CS05,Vel23}; and for the general
Alt--Phillips functional with $\gamma\in(0,2)$, see \cite{AP86}.

The purpose of this short article is to give a self-contained and pedagogical
proof of the optimal regularity for $\gamma\in(0,1)$. The proof is based
on a dichotomy argument inspired by \cite{DS20}, which originally applies to almost-minimizers. See also \cite{ara25} for a similar problem.

Roughly
speaking, the dichotomy says for the $L^2$ average of $u$ large enough, either the $\beta$-rescaled average
decays by $1/2$, or the
minimizer is close to a positive constant. Iterating this yields
either $\beta$-growth from a free-boundary point or, after stopping, a
Harnack-type estimate in a ball contained in the positivity set, with radius comparable to the pointwise value of $u$. The
optimal $C^{1,\beta-1}$ regularity then follows by combining the $\beta$-growth rate from the free boundary, and Harnack inequality in the positive set.

\section{Hölder Regularity}
In this section we prove a preliminary Hölder estimate that applies to minimizers of small
perturbations of the Dirichlet energy, whose potential has at most quadratic
growth. The proof follows Campanato's approach. 
\begin{proposition}\label{prop:Hölder_proposition}
    For any $\alpha \in (0, 1),$ and $K \geq 1$, there exists $\delta_0 \in (0,1)$ and $C >0$ depending only on $d, \alpha, K$ s.t. for any $u \in H^1(B_1)$, 
     and 
    \[\fint_{B_1} u^2 \leq 1,\] that minimizes
    \[\F_{\delta}(v, B_1) = \int_{B_1}|\nabla v|^2 + \delta F(x,v), \qquad 0 \leq F(x, v) \leq K(1 + v^2), \qquad 0 < \delta \leq \delta_0,\]
   among $v \in H^1(B_1)$ with same trace as $u$,
    then one has 
    \[\norm{u}_{C^{0, \alpha}(\overline{B_{1/2}})} \leq C.\]
\end{proposition}

First we need the following.
\begin{lemma}\label{lem:Hölder_lemma}
    For any $\alpha \in (0, 1)$, and $\Lambda \geq 1$,  there exists 
        \[
        \rho=\rho(d,\alpha)\in(0,1/2),\qquad
        M=M(d)\geq 1,\qquad
        \delta_1=\delta_1(d,\alpha,\Lambda)\in(0,1),
    \]
    such that  for any $w \in H^1(B_1)$,
    \begin{equation}\label{eq:Hölder_lemma_assumption}
        \fint_{B_1} w^2 \leq1 ,
    \end{equation}
    that minimizes 
    \[\G_\delta(v,B_1) = \int_{B_1} |\nabla v|^2 + \delta G(x, v), \qquad 0 \leq G(x, v) \leq \Lambda(1 + v^2), \qquad 0 < \delta \leq \delta_1,\]
     among $v \in H^1(B_1)$ with same trace as $w$, there exists $m \in \R$ s.t. $ |m| \leq M$ and 
     \begin{equation}\label{eq:Hölder_lemma}
         \frac{1}{\rho^{2\alpha}}\fint_{B_\rho} |w - m|^2  \leq 1.
     \end{equation}
\end{lemma}

\begin{proof}[Proof of \Cref{lem:Hölder_lemma}]
Constants $C=C(d)$ may change from line to line. Let $h$ be the harmonic
replacement of $w$ in $B_1$. Now
\[
    \int_{B_1} |\nabla w|^2+\delta G(x,w)\,dx
    \leq
    \int_{B_1} |\nabla h|^2+\delta G(x,h)\,dx .
\]
Since $h-w\in H_0^1(B_1)$ and $h$ is harmonic,
\[
    \int_{B_1} |\nabla w|^2\,dx
    =
    \int_{B_1} |\nabla h|^2\,dx
    +
    \int_{B_1} |\nabla(w-h)|^2\,dx .
\]
Therefore, using $G\geq0$,
\[
\begin{aligned}
    \int_{B_1}|\nabla(w-h)|^2\,dx
    &\leq
    \delta\int_{B_1}\bigl(G(x,h)-G(x,w)\bigr)\,dx \leq
    \Lambda\delta\int_{B_1}(1+h^2)\,dx .
\end{aligned}
\]
Also, using \eqref{eq:Hölder_lemma_assumption} and Poincaré
\[
    \int_{B_1} h^2\,dx
    \leq
    C\int_{B_1} w^2\,dx
    +
    C\int_{B_1} |w-h|^2\,dx
    \leq
    C+
    C\int_{B_1}|\nabla(w-h)|^2\,dx .
\]
Thus
\[
    \int_{B_1}|\nabla(w-h)|^2\,dx
    \leq
    C\Lambda\delta
    \left(
        1+\int_{B_1}|\nabla(w-h)|^2\,dx
    \right).
\]
Choosing $\delta_1\leq c(d)\Lambda^{-1}$, we absorb and obtain
\[
    \int_{B_1}|\nabla(w-h)|^2\,dx
    +
    \int_{B_1}|w-h|^2\,dx
    \leq
    C\Lambda\delta .
\]
Consequently,
\[
    \fint_{B_1} h^2\,dx\leq C .
\]
Since $h$ is harmonic, $h^2$ is subharmonic, and thus
\[
    |h(0)|^2\leq C\fint_{B_1}h^2\,dx\leq C .
\]
Set
\[
    m:=h(0).
\]
Thus $|m|\leq M$ for $M=M(d)$.

For $0<\rho<1/2$, using $h$ is harmonic
\[
    \fint_{B_\rho}|h-h(0)|^2\,dx
    \leq
    C\rho^2 \left(\fint_{B_1}h^2\,dx  + |h(0)|^2\right)
    \leq
    C\rho^2 .
\]
Hence
\[
\begin{aligned}
    \fint_{B_\rho}|w-m|^2\,dx
    &\leq
    C\fint_{B_\rho}|w-h|^2\,dx
    +
    C\fint_{B_\rho}|h-h(0)|^2\,dx       \leq
    C\rho^{-d}\Lambda\delta
    +
    C\rho^2 .
\end{aligned}
\]
Choose $\rho=\rho(d,\alpha)\in(0,1/2)$ so that
\[
    C\rho^2\leq \frac12\rho^{2\alpha}.
\]
Then choose $\delta_1=\delta_1(d,\alpha,\Lambda)$ smaller so that
\[
    C\rho^{-d}\Lambda\delta_1
    \leq
    \frac12\rho^{2\alpha}.
\]
\end{proof}

\begin{proof}[Proof of \Cref{prop:Hölder_proposition}]
Constants $C=C(d,\alpha,K)$ may change from line to line. Let $\rho$ and $M$
be as in \Cref{lem:Hölder_lemma}, and set
\[
    R:=\frac{M}{1-\rho^\alpha},
    \qquad
    \Lambda:=K(3+2R^2),
    \qquad
    \delta_0:=\delta_1(d,\alpha,\Lambda).
\]

We first prove the estimate at the origin. Define
\[
    u_0:=u,\qquad p_0:=0,\qquad
    F_0:=F.
\]
Following directly from \Cref{lem:Hölder_lemma}, one obtain $m_1 \in \R$ s.t. 
\[|m_1| \leq M, \qquad \fint_{B_\rho} |u_0 - m_1|^2\, dx \leq \rho^{2\alpha}.\]
The idea is to construct inductively 
\begin{align*}
     p_k&=\sum_{j=1}^k \rho^{(j-1)\alpha}m_j \in \R, \\
      u_k(x)&\defeq \rho^{-k\alpha}\bigl(u(\rho^kx)-p_k\bigr),\\
        F_k(x,z)
    &\defeq
    \rho^{2k(1-\alpha)}
    F\bigl(\rho^kx,p_k+\rho^{k\alpha}z\bigr).
\end{align*}
Assume the construction is done up to step $k$, with
\[\fint_{B_1} u_k^2\, dx \leq 1.\]
Since $|m_j|\leq M$,
\[
    |p_k|
    \leq
    M\sum_{j=1}^k \rho^{(j-1)\alpha}
    \leq R .
\]
Hence
\[
\begin{aligned}
    0\leq F_k(x,z)
    &\leq
    K\rho^{2k(1-\alpha)}
    \left(1+\left|p_k+\rho^{k\alpha}z\right|^2\right)  \leq
    K(1+2R^2+2z^2)
    \leq
    \Lambda(1+z^2).
\end{aligned}
\]
Moreover, by scaling, $u_k$ minimizes
\[
    \int_{B_1}|\nabla v|^2+\delta F_k(x,v)\,dx
\]
among competitors with the same trace as $u_k$. Thus \Cref{lem:Hölder_lemma} gives $m_{k+1}\in \R$ such that
\[
    |m_{k+1}|\leq M,
    \qquad
    \fint_{B_\rho}|u_k-m_{k+1}|^2\,dx
    \leq
    \rho^{2\alpha}.
\]
Set
\[
    p_{k+1}:=p_k+\rho^{k\alpha}m_{k+1}.
\]
Define
\[
    u_{k+1}(x)
    :=
    \rho^{-\alpha}\bigl(u_k(\rho x)-m_{k+1}\bigr)
    =
    \rho^{-(k+1)\alpha}
    \bigl(u(\rho^{k+1}x)-p_{k+1}\bigr).
\]
Then
\[
    \fint_{B_1}u_{k+1}^2\,dx
    =
    \rho^{-2\alpha}
    \fint_{B_\rho}|u_k-m_{k+1}|^2\,dx
    \leq1 .
\]
This closes the induction.

From the definition of $u_k$,
\[
    \fint_{B_{\rho^k}} |u-p_k|^2\,dx
    \leq
    \rho^{2k\alpha}.
\]
Since
\(
    |p_\ell-p_k|
    \leq
    M\sum_{j=k+1}^{\ell}\rho^{(j-1)\alpha},
\)
the sequence $p_k$ converges. Let
\[
    p_\infty:=\lim_{k\to\infty}p_k .
\]
Moreover,
\(
    |p_\infty-p_k|
    \leq
    C\rho^{k\alpha}.
\)
Therefore
\[
\begin{aligned}
    \left(\fint_{B_{\rho^k}}|u-p_\infty|^2\,dx\right)^{1/2}
    &\leq
    \left(\fint_{B_{\rho^k}}|u-p_k|^2\,dx\right)^{1/2}
    +
    |p_k-p_\infty|                                     \leq
    C\rho^{k\alpha}.
\end{aligned}
\]
If $0<r<1/2$, choose $k$ such that
\(
    \rho^{k+1}<r\leq \rho^k.
\)
Then
\[
\begin{aligned}
    \left(\fint_{B_r}|u-p_\infty|^2\,dx\right)^{1/2}
    &\leq
    C\left(\frac{\rho^k}{r}\right)^{d/2}
    \left(\fint_{B_{\rho^k}}|u-p_\infty|^2\,dx\right)^{1/2}  \leq
    C\rho^{k\alpha}
    \leq
    Cr^\alpha .
\end{aligned}
\]
Thus the precise representative satisfies
\[
    u(0):=p_\infty,
    \qquad
    |u(0)|\leq R,
\]
and
\begin{equation}\label{eq:origin_campanato_holder}
    \left(\fint_{B_r}|u-u(0)|^2\,dx\right)^{1/2}
    \leq
    Cr^\alpha,
    \qquad 0<r<1/2 .
\end{equation}
For any other point $x_0\in \overline{B_{1/2}}$, apply \eqref{eq:origin_campanato_holder} to
\[
    \widetilde u(y):=2^{-d/2}u\left(x_0+\frac y2\right),
    \qquad y\in B_1 .
\]
Thus by Campanato's characterization \cite[Theorem 3.1]{HL11},
\[
    \norm{u}_{C^{0,\alpha}(\overline{B_{1/2}})}
    \leq C .
\]
\end{proof}

As a result, one may apply \Cref{prop:Hölder_proposition} to minimizers of Alt--Phillips Problem and obtain $C^{0, \alpha}$ regularity.
\begin{corollary}\label{cor:Hölder_cor}
    Let $u \in H^1(B_1)$ minimize $\J^\gamma$ among $H^1(B_1)$ with same boundary data.  Then for any $\alpha \in (0, 1)$, $u \in C_{\text{loc}}^{0, \alpha}(B_1)$. In particular there exists $C> 0$ depending on $d, \alpha, \gamma$ s.t. 
    \[\norm{u}_{C^{0, \alpha}(\overline{B_{1/2}})} \leq C\left(1 + \norm{u}_{L^2(B_1)}\right).\]
\end{corollary}
\begin{proof}
Constants $C=C(d,\alpha,\gamma)$ may change from line to line. Take $A=A(d,\alpha,\gamma)\geq1$ large, and define
\[
    a:=A\left(1+\norm{u}_{L^2(B_1)}\right),
    \qquad
    \widetilde u:=\frac{u}{a},
\]
 so that
\[
    \fint_{B_1}\widetilde u^2\,dx
    \leq1 .
\]
Dividing the energy by $a^2$, we see that $\widetilde u$ minimizes
\[
    \int_{B_1}|\nabla v|^2+  \delta F_a(v)  \,dx, \qquad   \delta:=a^{-1},
    \qquad
    F_a(v):=a^{\gamma-1}v_+^\gamma,
\]
among competitors with the same trace as
$\widetilde u$. Since $a\geq1$ and $0<\gamma<1$,
\[
    0\leq F_a(v)
    \leq
    v_+^\gamma
    \leq
    1+v^2 .
\]
Choose $A$ larger if necessary so that
\[
    \delta=a^{-1}\leq A^{-1}\leq \delta_0,
\]
where $\delta_0$ is the constant in \Cref{prop:Hölder_proposition} with $K=1$.
Thus \Cref{prop:Hölder_proposition} gives
\[
    \norm{\widetilde u}_{C^{0,\alpha}(\overline{B_{1/2}})}
    \leq C .
\]
Multiplying by $a=A\left(1+\norm{u}_{L^2(B_1)}\right)$ yields
\[
    \norm{u}_{C^{0,\alpha}(\overline{B_{1/2}})}
    \leq
    C\left(1+\norm{u}_{L^2(B_1)}\right).
\]
The same argument on balls $B_r(x_0)\Subset B_1$ gives
\(
    u\in C^{0,\alpha}_{\mathrm{loc}}(B_1).
\)
\end{proof}

\section{Optimal $C^{1, \beta -1}$ Regularity}

We state the main goal of the article.
\begin{theorem}\label{thm:optimal_regularity}
     Let $u \in H^1(B_1)$, $u \geq 0$, minimize $\J^\gamma$ among $H^1(B_1)$ with same boundary data. Then $u \in C^{1, \beta - 1}_{\text{loc}}(B_1)$. In particular there exists $C > 0$ depending on $d, \gamma$ s.t. 
     \[\norm{u}_{C^{1, \beta - 1}(\overline{B_{1/2}})} \leq C\left(1 + \norm{u}_{L^2(B_1)}\right).\]
\end{theorem}

In the following we develop the key Dichotomy Argument. 

\begin{proposition}[Dichotomy Argument]\label{prop:dichotomy_argument}
    For any $\ep \in (0, 1)$, and $\xi \in (0, 1/2]$, there exists 
    \[\eta = \eta(d, \gamma, \ep) \in (0, 1/2), \qquad a_* = a_*(d, \gamma, \ep, \xi) \geq 1, \qquad c_0 =c_0(d, \gamma, \xi) \in (0, 1), \qquad C_0 = C_0(d, \gamma) \geq 1,\]
    s.t.  for any $u \in H^1(B_1)$, $u \geq 0$, s.t.
    \begin{equation}\label{eq:dichotomy_argument_assumption}
        \left(\fint_{B_1} u^2 \right)^{\frac{1}{2}}  = a \geq a_*,
    \end{equation}
    that minimizes $\J^\gamma$ among $H^1(B_1)$ with same boundary data,
    either 
    \begin{equation}\label{eq:dichotomy_argument_1}
        \xi^{-\beta} \left(\fint_{B_\xi} u^2\right)^{\frac{1}{2}} \leq \frac{1}{2}a,
    \end{equation}
    or there exists $m \in \R$, $m >0$ s.t. 
    \begin{equation}\label{eq:dichotomy_argument_2_1}
        c_0 a\leq m \leq C_0 a,
    \end{equation}
    and 
    \begin{equation}\label{eq:dichotomy_argument_2_2}
        \left(\fint_{B_\eta}|u - m|^2\, dx\right)^{\frac{1}{2}} \leq \ep a.
    \end{equation}
\end{proposition}
\begin{proof}
Constants $C=C(d,\gamma)$ may change from line to line. Set
\[
    \tilde{u}:=\frac{u}{a}.
\]
Then
\[
    \fint_{B_1}\tilde{u}^2\,dx=1,
\]
and $\tilde{u}$ minimizes
\[
    \int_{B_1}|\nabla v|^2+\delta F_a(v)\,dx,
    \qquad
    \delta:=a^{-1},
    \qquad
    F_a(v):=a^{\gamma-1}v_+^\gamma,
\]
among competitors with same trace as $\tilde{u}$. Since $a\geq1$ and $0<\gamma<1$,
\[
    0\leq F_a(v)\leq v_+^\gamma\leq 1+v^2.
\]

Let $h$ be the harmonic replacement of $\tilde{u}$ in $B_1$, then $h\geq0$. By minimality and orthogonality,
\[
\begin{aligned}
    \int_{B_1}|\nabla(\tilde{u}-h)|^2\,dx
    & = \int_{B_1}|\nabla \tilde{u}|^2 - |\nabla h|^2\leq
    \delta\int_{B_1}\bigl(F_a(h)-F_a(\tilde{u})\bigr)\,dx   \leq
    \delta\int_{B_1}(1+h^2)\,dx .
\end{aligned}
\]
Moreover, using \eqref{eq:dichotomy_argument_assumption} and Poincaré
\[
    \int_{B_1}h^2\,dx
    \leq
    C\int_{B_1}\tilde{u}^2\,dx
    +
    C\int_{B_1}|\tilde{u}-h|^2\,dx
    \leq
    C+
    C\int_{B_1}|\nabla(\tilde{u}-h)|^2\,dx .
\]
Taking $a_*$ large, hence $\delta=a^{-1}$ small, we absorb and get
\begin{equation}\label{eq:dichotomy_harmonic_close}
    \int_{B_1}|\nabla(\tilde{u}-h)|^2\,dx
    +
    \int_{B_1}|\tilde{u}-h|^2\,dx
    \leq
    C a^{-1}.
\end{equation}
In particular,
\[
    \fint_{B_1}h^2\,dx\leq C,
    \qquad
    0\leq h(0)\leq C.
\]
Set
\[
    m:=ah(0).
\]
Then $m\geq0$, and
\[
    m\leq C_0a.
\]

We distinguish two cases.

First assume
\[
    h(0)\leq c_0,
\]
where $c_0=c_0(d,\gamma,\xi)>0$ will be fixed. Since $h\geq0$ is harmonic,
interior estimates and Harnack gives
\[
    \fint_{B_\xi} |h - h(0)|^2\,dx
    \leq \norm{\nabla h}_{L^\infty(B_\xi)}^2 \xi^2 \leq C\norm{h}_{L^\infty(B_{3\xi/2})}^2 \leq 
    C h(0)^2
    \leq
    Cc_0^2 .
\]
Together with \eqref{eq:dichotomy_harmonic_close},
\[
\begin{aligned}
    \fint_{B_\xi}\tilde{u}^2\,dx
    &\leq
    C\fint_{B_\xi}|\tilde{u}-h|^2\,dx
    +
    C\fint_{B_\xi}|h - h(0)|^2\,dx  + C |h(0)|^2      \leq
    C\xi^{-d}a^{-1}+Cc_0^2 .
\end{aligned}
\]
Thus
\[
    \xi^{-2\beta}\fint_{B_\xi}u^2\,dx
    =
    a^2\xi^{-2\beta}\fint_{B_\xi}\tilde{u}^2\,dx
    \leq
    C\bigl(\xi^{-d-2\beta}a^{-1}+c_0^2\xi^{-2\beta}\bigr)a^2.
\]
Choose $c_0=c_0(d,\gamma,\xi)$ so that
\[
    Cc_0^2\xi^{-2\beta}\leq \frac18,
\]
and then choose $a_*=a_*(d,\gamma,\ep,\xi)$ large so that
\[
    C\xi^{-d-2\beta}a_*^{-1}\leq \frac18 .
\]
Then
\[
    \xi^{-\beta}
    \left(\fint_{B_\xi}u^2\,dx\right)^{1/2}
    \leq
    \frac12 a,
\]
which is \eqref{eq:dichotomy_argument_1}.

Now assume
\[
    h(0)>c_0.
\]
Then
\[
    c_0a\leq m\leq C_0a.
\]
For $0<\eta<1/2$,
\[
    \fint_{B_\eta}|h-h(0)|^2\,dx
    \leq
    C\eta^2\fint_{B_1}h^2\,dx
    \leq
    C\eta^2.
\]
Using \eqref{eq:dichotomy_harmonic_close},
\[
\begin{aligned}
    \fint_{B_\eta}|u-m|^2\,dx
    &=
    a^2\fint_{B_\eta}|\tilde{u}-h(0)|^2\,dx      \leq
    Ca^2\fint_{B_\eta}|\tilde{u}-h|^2\,dx
    +
    Ca^2\fint_{B_\eta}|h-h(0)|^2\,dx      \leq
    C\bigl(\eta^{-d}a^{-1}+\eta^2\bigr)a^2 .
\end{aligned}
\]
Choose $\eta=\eta(d,\gamma,\ep)\in(0,1/2)$ so that
\[
    C\eta^2\leq \frac12\ep^2,
\]
and increase $a_*$ so that
\[
    C\eta^{-d}a_*^{-1}\leq \frac12\ep^2.
\]
Then
\[
    \left(\fint_{B_\eta}|u-m|^2\,dx\right)^{1/2}
    \leq
    \ep a.
\]
This proves \eqref{eq:dichotomy_argument_2_1}--\eqref{eq:dichotomy_argument_2_2}.
\end{proof}

Under the second alternative \eqref{eq:dichotomy_argument_2_2}, we conduct $\alpha$-rescaling to stay a positive distance away from $0$.
\begin{lemma}\label{lem:staying_away_from_zero}
    Let \(0<c_0\leq 1\leq C_0\). There exist
    \[
        \ep_*=\ep_*(d,\gamma,c_0,C_0)\in(0,1),
        \qquad
        \delta_*=\delta_*(d,\gamma,c_0,C_0)\in(0,1),
    \]
    such that for any $u \in H^1(B_1)$ that minimizes
    \[
        \int_{B_1}|\nabla v|^2+\delta F(x,v)\,dx, \qquad   0<\delta\leq \delta_*,
        \qquad
        0\leq F(x,v)\leq C(d,\gamma)(1+v^2),
    \]
    among $v \in H^1(B_1)$ with same boundary data, and satisfies for some \(m_0\in\R\),
      \[
        \left(\fint_{B_1}|u-m_0|^2\,dx\right)^{1/2}\leq \ep_*,
        \qquad
        c_0\leq m_0\leq C_0 ,
    \]
    one has 
    \begin{equation}\label{eq:positive_distance_claim}
        \frac14 c_0\leq u(x)\leq 4C_0
        \qquad \text{in }B_{1/2}.
    \end{equation}
\end{lemma}
\begin{proof}
Constants \(C=C(d,\gamma,c_0,C_0)\) may change from line to line. We apply
\Cref{prop:Hölder_proposition} with exponent \(1/2\).

Let
\[
    \tilde{u}:=\frac{u-m_0}{\ep_*}, \qquad 
    \fint_{B_1} \tilde{u}^2\,dx\leq1.
\]
Moreover \( \tilde{u}\) minimizes
\[
    \int_{B_1}|\nabla v|^2+\mu \widetilde F(x,v)\,dx,
    \qquad
    \mu:=\delta\ep_*^{-2},
    \qquad
    \widetilde F(x,v):=F(x,m_0+\ep_* v),
\]
among competitors with the same trace as \(\tilde{u}\).
Since \(m_0\leq C_0\) and \(\ep_*\leq1\),
\[
\begin{aligned}
    0\leq \widetilde F(x, v)
    &\leq
    C(d,\gamma)\left(1+|m_0+\ep_* v|^2\right)  \leq
    C(d,\gamma,C_0)(1+v^2).
\end{aligned}
\]
Let \(\delta_0\) and \(C_1\) be the constants from
\Cref{prop:Hölder_proposition}, with \(K=C(d,\gamma,C_0)\) and Hölder exponent
\(1/2\). Choose
\[
    \ep_*>0
    \qquad \text{so small that} \qquad
    C_1\ep_*\leq \frac34 c_0,
    \qquad
    C_1\ep_*\leq 3C_0.
\]
Then choose
\[
    \delta_*:=\delta_0\ep_*^2.
\]
Thus
\[
    0<\mu=\delta\ep_*^{-2}\leq \delta_0,
\]
and one may apply \Cref{prop:Hölder_proposition}, so that
\[
    \norm{ \tilde{u}}_{C^{0,1/2}(\overline{B_{1/2}})}
    \leq C_1.
\]
Hence in particular
\[
    |u-m_0|\leq C_1\ep_*
    \qquad \text{in }B_{1/2}.
\]
Therefore, in \(B_{1/2}\),
\[
   \frac14c_0\leq m_0-C_1\ep_*\leq 
    u(x)\leq m_0+C_1\ep_*
    \leq 
    4C_0.
\]
\end{proof}

Now we're ready to prove the main $\beta$-growth rate and Harnack type Inequality.
\begin{proposition}[Harnack Inequality and $\beta$-growth]\label{prop:Harnack_and_beta_growth}
    There exists $c_H \in (0, 1)$, $C_H \geq 1$ and $\eta_H \in (0, 1/2)$ depending only on $d, \gamma$ s.t. for any $u \in H^1(B_1)$, $u \geq 0$ and 
    \[\fint_{B_1} u^2 \leq 1,\] that minimizes $\J^\gamma$ among $H^1(B_1)$ with same boundary data, either $u(0) > 0$ and 
    \begin{equation}\label{eq:Harnack}
        c_H u(0) \leq u(x) \leq C_H u(0), \qquad \forall \ |x| \leq \eta_H u(0)^{\frac{1}{\beta}},
    \end{equation}
    or $u(0) = 0$ and 
    \begin{equation}\label{eq:beta_growth}
        \left(\fint_{B_r} u^2\right)^{\frac{1}{2}} \leq C_H r^\beta, \qquad \forall \ 0 < r < 1/2.
    \end{equation}
\end{proposition}
\begin{proof}
Constants \(C=C(d,\gamma)\) may change from line to line. Note $d, \gamma, C_0$ are universal. Fix \(
    \xi:=\frac12, 
\)
thus \(c_0 \) from \Cref{prop:dichotomy_argument}. Decreasing \(c_0\) and increasing \(C_0\) if necessary, assume
\(
    0<c_0\leq1\leq C_0 .
\)
Thus we fix \(\ep_*,\delta_*\) as in
\Cref{lem:staying_away_from_zero}. Apply \Cref{prop:dichotomy_argument} with
\(
    \ep=\ep_*,\
    \xi=\frac12,
\)
and thus we fix \(\eta = \eta(\ep_*)\in(0,1/2)\), \(a_* = a_*(\ep_*, \xi)\geq1\). Set
\[
    A:=\max\{a_*,\delta_*^{-1},1\}.
\]
Therefore all the above are universal, depending only on $d, \gamma$.

For \(0<r<1\), define
\[
    u_r(x):=r^{-\beta}u(rx),
    \qquad
    a(r):=
    r^{-\beta}
    \left(\fint_{B_r}u^2\,dx\right)^{1/2}
    =
    \left(\fint_{B_1}u_r^2\,dx\right)^{1/2}.
\]
Then \(u_r\) minimizes \(\J^\gamma\) in \(B_1\) among $H^1(B_1)$ with same boundary data. We work with sequence 
Let
\[
    r_k:=\xi^k = 2^{-k},
    \qquad
    a_k:=a(r_k).
\]
Notice there exists universal constant $\overline{C} = \overline{C}(d, \gamma) \geq 1$ s.t.
\begin{equation}\label{eq:dyadic_growth_bound}
    a_{k+1}\leq \overline{C}a_k , \qquad \forall \ k \in \N.
\end{equation}
Indeed,
\[
\begin{aligned}
    a_{k+1}
    &=
    r_{k+1}^{-\beta}
    \left(\fint_{B_{r_{k+1}}}u^2\,dx\right)^{1/2}       \leq
    \xi^{-\beta}r_k^{-\beta}
    \xi^{-d/2}
    \left(\fint_{B_{r_k}}u^2\,dx\right)^{1/2}
    =
    \overline{C} a_k .
\end{aligned}
\]

 Define
\[
    \mathcal K
    :=
    \left\{
        k\geq0:
        a_k\leq \overline{C} A+2^{-k}
    \right\}.
\]
Since \(\left(\fint_{B_1}u^2\right)^{\frac{1}{2}}=a_0\leq1\), one has \(0\in\mathcal K\).

Assume first
\[
    \mathcal K=\{0,1,2,\ldots\}.
\]Then \(a_k\leq C\) for every \(k\). If \(r_{k+1}<r\leq r_k\), then
\[
\begin{aligned}
    r^{-\beta}
    \left(\fint_{B_r}u^2\,dx\right)^{1/2}
    &\leq
    \left(\frac{r_k}{r}\right)^{\beta+d/2}
    r_k^{-\beta}
    \left(\fint_{B_{r_k}}u^2\,dx\right)^{1/2}      \leq
    C a_k
    \leq C .
\end{aligned}
\]
Thus \eqref{eq:beta_growth} holds. By Lebesgue Differentiation one has $u(0) = 0$.

Otherwise,  let \(k\) be the first index
such that
\[
    k\in\mathcal K,
    \qquad
    k+1\notin\mathcal K .
\]
Then \(a_k>A\). Indeed, if \(a_k\leq A\), then by
\eqref{eq:dyadic_growth_bound},
\[
    a_{k+1}\leq \overline{C}a_k \leq  \overline{C}A
    \leq
    \overline{C} A +2^{-(k+1)},
\]
contradicting \(k+1\notin\mathcal K\). Now apply \Cref{prop:dichotomy_argument} to \(u_{r_k}\). Since
\[
    \left(\fint_{B_1}u_{r_k}^2\,dx\right)^{1/2}
    =
    a_k>A\geq a_*,
\]
the dichotomy applies.

If the first alternative holds, then
\[
    a_{k+1}
    =
    \xi^{-\beta}
    \left(\fint_{B_\xi}u_{r_k}^2\,dx\right)^{1/2}
    \leq
    \frac12 a_k .
\]
Since \(k\in\mathcal K\),
\[
    a_{k+1}
    \leq
    \frac12(\overline{C} A+2^{-k})
    \leq
\overline{C}A+2^{-(k+1)},
\]
again contradicting \(k+1\notin\mathcal K\). Hence the second alternative holds: there exists \(m\in \R\) such that
\[
    c_0a_k\leq m\leq C_0a_k,
    \qquad
    \left(\fint_{B_\eta}|u_{r_k}-m|^2\,dx\right)^{1/2}
    \leq
    \ep_*a_k .
\]

Set
\[
    \tilde{u}_k(y):=\frac{u_{r_k}(\eta y)}{a_k},
    \qquad
    m_0:=\frac{m}{a_k}.
\]Then
\[
    c_0\leq m_0\leq C_0,
    \qquad
    \left(\fint_{B_1}|\tilde{u}_k-m_0|^2\,dy\right)^{1/2}
    \leq
    \ep_* .
\]
Moreover \(\tilde{u}_k\) minimizes
\[
    \int_{B_1}|\nabla v|^2+\delta_k F_k(v)\,dy,
    \qquad
    \delta_k:=a_k^{-1}, \qquad
    F_k(v):=\eta^2 a_k^{\gamma-1}v_+^\gamma .
\]Since \(a_k\geq A\geq1\),
\[
    0\leq F_k(v)\leq v_+^\gamma\leq 1+v^2,
    \qquad
    0<\delta_k=a_k^{-1}\leq A^{-1}\leq \delta_* .
\]
Thus \Cref{lem:staying_away_from_zero} gives
\[
    \frac14c_0\leq \tilde{u}_k(y)\leq 4C_0
    \qquad\text{in }B_{1/2}.
\]
Scaling back,
\begin{equation}\label{eq:positive_after_stopping}
    ca_k\leq u_{r_k}(x)\leq Ca_k
    \qquad\text{for } x\in B_{\eta/2}.
\end{equation}
Note $u(0) = 0$ is impossible otherwise $a_k = 0$ contradicting $k+1 \notin \K$. Thus $u(0) > 0$, and in particular
\begin{equation}\label{eq:ak_u0_comparable}
    ca_k\leq r_k^{-\beta}u(0)\leq Ca_k.
\end{equation}

Since \(k\in\mathcal K\),
\[
    a_k\leq \overline{C} A+2^{-k}\leq C.
\]
Therefore, from \eqref{eq:ak_u0_comparable},
\[
    r_k\geq c\,u(0)^{1/\beta}.
\]
Choose \(\eta_H\in(0,1/2)\), depending only on \(d,\gamma\), so small that
\[
    \eta_H u(0)^{1/\beta}
    \leq
    \frac{\eta}{2} r_k .
\]
Therefore using \eqref{eq:positive_after_stopping},
\[
    cr_k^\beta a_k
    \leq
    u(x)
    \leq
    Cr_k^\beta a_k, \qquad  |x| \leq    \eta_H u(0)^{1/\beta}.
\]
Using \eqref{eq:ak_u0_comparable} once more one obtain
\[
    c_Hu(0)\leq u(x)\leq C_Hu(0),
    \qquad
    |x|\leq \eta_Hu(0)^{1/\beta}.
\]
\end{proof}

Finally we're able to conclude the proof of Optimal Regularity.

\begin{proof}[Proof of \Cref{thm:optimal_regularity}]
   Constants $C=C(d,\gamma)$ may change from line to line. Dividing $u$ by a universal multiple of $1 + \norm{u}_{L^2(B_1)}$ so that $\fint_{B_1} u^2 \leq 1$ and it minimizes energy of the form \[
    \mathcal J^\gamma_\lambda(v)
    :=
    \int_{B_1} |\nabla v|^2+\lambda v_+^\gamma,
    \qquad 0\le \lambda\le1.
\] 
   All preceding propositions remain valid for this $\J^\gamma_\lambda$ with the same constants.  For any $x \in \{u > 0\} \cap B_{1/2}$, using \eqref{eq:Harnack}, we work with $B_\rho(x) \Subset \{u > 0\}$ where $\rho = \frac{1}{2}\eta_H u(x)^{\frac{1}{\beta}}$. Recall from \cite{AP86} that $u$ solves 
   \[\D u = \frac{\gamma}{2}u^{\gamma - 1} \chi_{\{u > 0\}}.\] Then using interior elliptic estimate (see \cite{FRRO22}) in the positive set
   \begin{align*}
       \norm{\nabla u}_{C^{0, \beta - 1}(B_{\rho/2}(x))} &\leq C \frac{1}{\rho^{\beta}} \left(\norm{u}_{L^\infty(B_\rho(x))} + \rho^2\norm{u^{\gamma -1}}_{L^\infty(B_\rho(x))}\right) \leq C\left(\rho^{-\beta}u(x) + \rho^{2 - \beta} u(x)^{\gamma -1}\right)\\
       &\leq C\left(1 + \rho^{2 - \beta + \beta(\gamma - 1)}\right)= C.
   \end{align*}
   For points $x \in \p\{u > 0\} \cap B_{1/2}$, one directly apply $\beta$-growth rate \eqref{eq:beta_growth}. Combining both, one conclude using Campanato's Method.
\end{proof}

\section*{Acknowledgments}
The author would like to thank Daniela De Silva and Ovidiu Savin for the fruitful discussions.

\bibliographystyle{alpha}
\bibliography{references}

\end{document}